%% file: bgsf5.tex
\documentclass[12pt,a4paper]{amsart}
\usepackage[T1]{fontenc}
\usepackage{amssymb}
\usepackage{amsmath}
\usepackage{amsfonts}
\usepackage{ifthen}

\usepackage[all]{xy}

\newcommand{\figs}{\usepackage{graphicx} \usepackage{color} \graphicspath{{figs/}} \numberwithin{figure}{section} }

\figs

\newcounter{margnotes}

\def\sideremark#1{\ifvmode\leavevmode\fi\vadjust{\vbox to0pt{\vss % the remark
      \hbox to 0pt{\hskip\hsize\hskip1em           %                will appear only
 \vbox{\hsize3cm\tiny\raggedright\pretolerance10000%                on the side
 \noindent #1\hfill}\hss}\vbox to8pt{\vfil}\vss}}}%
\newcounter{lemenumi}
\newcommand{\labelemenumi}{(\alph{lemenumi})}

\newtheorem{theorem}{Theorem}[section]

\newtheorem{proposition}[theorem]{Proposition}
\newtheorem{corollary}[theorem]{Corollary}

\theoremstyle{definition}
\newtheorem{definition}[theorem]{Definition}

\theoremstyle{remark}
\newtheorem{remark}[theorem]{Remark}

\numberwithin{equation}{section}

\newcommand{\bbC}{\mathbb{C}}
\newcommand{\bbR}{\mathbb{R}}

\newcommand{\bbQ}{\mathbb{Q}}

\newcommand{\dr}{\mathrm{d}}

\newcommand{\ep}{\varepsilon}

\newcommand{\dul}{\mathcal{D}}

\title[Finite cyclicity of  slow-fast Darboux systems]{Finite cyclicity of quadratic slow-fast Darboux systems with a  two-saddle loop}
\author{Marcin Bobie\'nski}
\address{Institute of Mathematics, Warsaw University,
ul. Banacha 2, 02-097 Warsaw, Poland}
\email{mbobi@mimuw.edu.pl}
\thanks{Supported by Polish NCN Grant No 2011/03/B/ST1/00330}
\author{Lubomir Gavrilov}
\address{Institut de Math\'{e}matiques de Toulouse, UMR 5219\\
 Universit\'{e} de Toulouse, CNRS\\
 UPS IMT, F-31062 Toulouse Cedex 9,  France}
\date{\today}

\begin{document}
\begin{abstract}
We prove that the cyclicity of a  quadratic slow-fast integrable system  of Darboux type with a double heteroclinic loop, is finite and uniformly bounded.
\end{abstract}
\maketitle
\section{Introduction}
\label{sec:intro}
%This paper is a contribution to the following existential infinitesimal 16th Hilbert problem: find a bound for the  number of limit cycles, which a perturbed plane polynomial %foliation ${\mathcal F}_\delta$ can have. 

Let ${\mathcal F}_\varepsilon$,  be an analytic family of analytic real plane foliations (or vector fields), depending on a small parameter $\varepsilon$, and having for all $\varepsilon>0$ a bounded period annulus (a nest of periodic  orbits) $\Pi_\varepsilon$. Consider a further multi parameter analytic deformation ${\mathcal F}_{\varepsilon,\delta}$ of ${\mathcal F}_\varepsilon$, and denote by $Z(\varepsilon)=Cycl({\mathcal F}_{\varepsilon,\delta},\bar{\Pi}_{\varepsilon})$  the maximal number of limit cycles which bifurcate from the closure $\bar{\Pi}_{\varepsilon}$ for sufficiently small $\|\delta \|$. The number $Z(\varepsilon)$ is therefore the cyclicity of the closed   period annulus $\bar{\Pi}_\varepsilon$ with respect to the deformed foliation ${\mathcal F}_{\varepsilon,\delta}$ \cite{rous98}. The purpose of the paper is to prove the $\varepsilon$-uniform boundedness of $Z(\varepsilon)$ in the case, when ${\mathcal F}_\varepsilon$, $\varepsilon>0$, is a Darboux integrable plane foliation, and ${\mathcal F}_0$ has a curve of singular points (slow manifold), as on fig. \ref{fig_slowfast}.
The family ${\mathcal F}_\varepsilon$ will be referred to as a \emph{slow-fast integrable system of Darboux type}.

\begin{figure}[htpb]
\def\svgwidth{8cm}
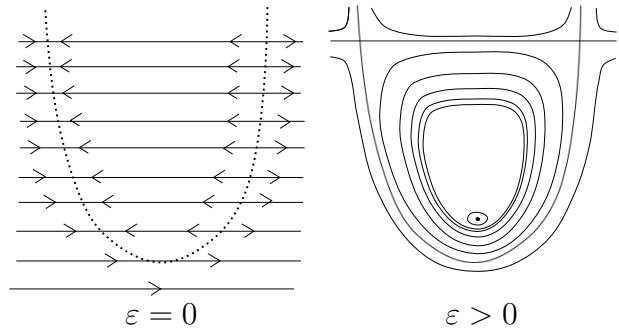
\caption{The slow-fast    Darboux foliation ${\mathcal F}_{\varepsilon}$.}
\label{fig_slowfast}
\end{figure}

Another motivation for our result comes from the program of proving uniform finiteness of the number  of limit cycles of plane quadratic vector fields \cite{drr94,rous98} (the existential Hilbert's 16th problem). Indeed, a family of plane quadratic systems can be slow-fast, with a degenerate graphics as on   fig. \ref{fig_slowfast}. 

The main difficulty in the study of ${\mathcal F}_{\varepsilon,\delta}$ is on  the one hand the "turning point" of ${\mathcal F}_{0}$, and on the other hand the double heteroclinic loop of ${\mathcal F}_{\varepsilon}$ ($\varepsilon > 0$). The zeros of pseudo-Abelian integrals appearing in the usual Poinacr\'e-Melnikov method does not detect  the so called alien limit cycles \cite{duro06}. Nevertheless, in  recent papers \cite{bmn09, bmn13}, it  has been shown that the number of zeros of pseudo Abelian integrals, which arise as a first approximation of the first return map of our slow-fast system, have the desired uniform finiteness property. 

In this paper  we  replace   the pseudo-Abelian integrals studied in \cite{bmn09, bmn13} by the true Dulac maps defining the limit cycles as fixed points.
Our main result is the uniform finiteness of the number of limit cycles  of the slow-fast Darboux system under consideration. The analytic deformations
${\mathcal F}_{\varepsilon,\delta}$ which we consider are arbitrary and can depend on an arbitrary given number of parameters.

Our method makes a strong use of the properties of the foliation ${\mathcal F}_{\varepsilon,\delta}$ in a complex domain, where we apply a technique derived from the so called  "Petrov trick", along the lines of \cite{gavr11,gavr11a}. 

\section{Statement of the result}

Using the notations of \cite{bmn13}, let $P_0=P_0(x,y),  P_1=P_1(x,y) $ be real bivariate polynomials
and consider the differential system
%\left{
\begin{equation}
\label{intfol}
P_1' = \varepsilon P_1 , \quad
P_0' = - P_0, \;\; \varepsilon \in \bbR^+ 
\end{equation}
which induces the foliation
\begin{equation}
\label{foliation}
{\mathcal F}_{\varepsilon}:  \varepsilon P_1 dP_0 + P_0 dP_1  = 0.
\end{equation}
It has a Darboux type first integral $H= P_0^\varepsilon P_1$, and  for $\varepsilon = 0$ a curve of singular points $\{(x,y): P_0(x,y)=0\}$. 
For $\varepsilon$ close to zero the curve is the slow manifold of the slow-fast system (\ref{intfol}). In the present paper we are interested in the simplest possible slow-fast Darboux system, shown on fig.\ref{fig_slowfast}. 
More precisely, assume that
\[
P_0=y-x^2,\ P_1=1-y. 
\]
Consider further the following perturbed slow-fast Darboux integrable foliation
\begin{equation}
{\mathcal F}_{\varepsilon,\delta}:  \varepsilon P_1 dP_0 + P_0 dP_1 +   P\dr x+Q \dr y = 0.
\end{equation}
where $P=P(x,y,\delta)$, $Q= Q(x,y,\delta)$ are real polynomials in $x,y$ depending analytically on $\delta\in (\bbR^N,0)$, and such that 
$$P(x,y,0)=Q(x,y,0)\equiv 0.$$
 For every fixed sufficiently small $\varepsilon>0$, denote by $Z(\varepsilon)$ the maximal number of limit cycles of ${\mathcal F}_{\varepsilon,\delta}$, which bifurcate from  the compact region $\bar{\Pi}$ bound by the curves $\{P_0=0\}$, $\{P_1=0\}$,
 for sufficiently small $\|\delta\|$. The number $Z(\varepsilon)$ is therefore the cyclicity of the closed period annulus  $\bar{\Pi}$  of ${\mathcal F}_{\varepsilon,0}$ under the deformation ${\mathcal F}_{\varepsilon,\delta}$.
The main result of the paper is the following theorem
\begin{theorem}
\label{th:main}
The cyclicity $Z(\varepsilon)$ is finite and uniformly bounded in $\varepsilon>0$.
\end{theorem}

\begin{remark}
Authors expect that the theorem \ref{th:main} remins true in the more general case, under the following assumptions on the unperturbed Darboux integrable foliation \ref{foliation}. Assume that the real curves $\{P_0=0\}$, $\{P_1=0\}$ are smooth, intersect transversally and bound a compact region $\overline{\Pi}$  in which the foliation (\ref{foliation}) has no singular points. Assume also that $\{P_0=0\}$ is transverse to the foliation $dP_1=0$ except at one point of simple tangency. Additional tools are needed to investigate the isoclines structures in this more general case.
\end{remark}

\section{The integrable foliation ${\mathcal F}_{\varepsilon,0}$ where $\varepsilon >0$}
%Analytic properties of the Dulac maps $\dul_j$ of the integrable case}
\label{sec:integrable}
\begin{figure}[htpb]
\def\svgwidth{8cm}
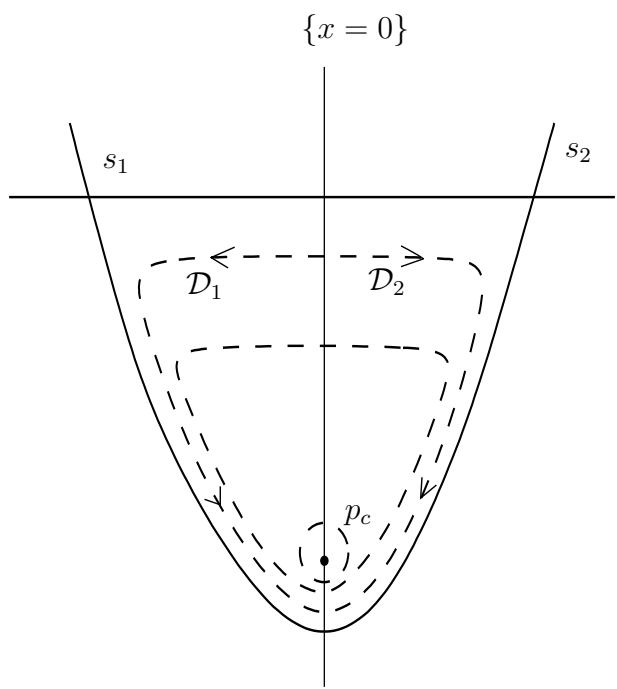
\caption{The real integrable plane foliation ${\mathcal F}_{\varepsilon,0}$ for $\varepsilon > 0$.}
\label{fig_fig21}
\end{figure}
In this section we assume that $\delta=0$ and $\varepsilon >0$ is a sufficiently small fixed real parameter. The foliation 
${\mathcal F}_{\varepsilon,0}$ has a Darboux type first integral $H=(y-x^2)^\varepsilon (1-y)$ and its phase portrait is shown on 
fig.\ref{fig_fig21}.  It has a real nest of cycles bounded by the parabola $y-x^2=0$ and a line $y-1=0$. The center point is located at 
$$p_c=(0,y_c), \; y_c =\tfrac{\ep}{\ep + 1}.
$$
\subsection{The complex leaves of ${\mathcal F}_{\varepsilon,0}$ }
\begin{figure}[htpb]
\input{figs/leafcov.pspdftex}
\caption{Every complex leaf  of the foliation ${\mathcal F}_{\varepsilon,0}$, which is not the line $\{y=1\}$,  is a double covering of the Riemann surface $C_\ep$.} 
\label{fig_contour}
\end{figure}
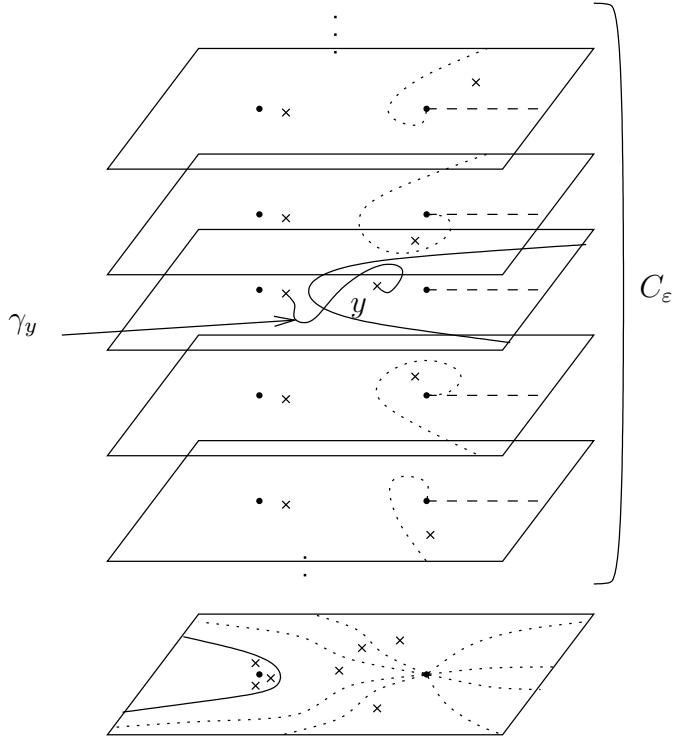
The leaves of dimension one of ${\mathcal F}_{\varepsilon,0}$ are connected open Riemann surfaces, on which the function 
$$(y-x^2)^\varepsilon (1-y)$$
is constant.  Every leaf intersects the cross-section $\{x=0\}$ at at least one point $(0,y_0)$.
The Riemann surface $C_{\ep}$ of the multi-valued analytic function 
\begin{equation}
\label{epfunct}
f(y):= y (1-y)^{1/\ep},  \; y\neq 1 
\end{equation}
is conformally equivalent either to $\mathbb C$ ($\varepsilon \not\in \bbQ$), or to $\mathbb C ^*$ ($\varepsilon \in \bbQ$) and on the leaf through $(0,y_0)$ holds
 \begin{equation}
\label{covering}
x^2= (f(y)-f(y_0))(1-y)^{-1/\ep}.
\end{equation}
This implies the following
\begin{proposition}
\label{pr22}
Every complex leaf of the integrable foliation ${\mathcal F}_{\varepsilon,0}$ which does not contain the center point $(0,y_c)$, or is not the line $\{y=1\}$, is a double covering of $C_\ep$ given by the formula (\ref{covering}).
The $y$-coordinates of the ramification points of the covering are the solution of the equation $f(y)=f(y_0)$. 
The leaf of ${\mathcal F}_{\varepsilon,0}$ through the center point $p_c=(0,y_c)$ has a single singular point at $p_c$ which is a normal crossing, and otherwise is a double   covering of $C_\ep$ given by the formula (\ref{covering}), with ramification points $(0,y)$ satisfying $f(y)=f(y_c)$. 
%The curve 
%$\{y: |f(y)|=|f(y_c)| \}$ containing the ramification points is shown on fig.\ref{fig_csingc}.
\end{proposition}

\subsection{Analytic continuation of the Dulac maps}
 Consider the cross-section $\{x=0\}$ parameterized by $y$, as well the two  Dulac maps $\dul_1=\dul_2$ 
\begin{equation}
\label{realdulac}
\dul_{1,2} : ( y_c,1) \rightarrow (0, y_c) 
\end{equation}
shown on figure \ref{fig_fig21}. 
One obvious  extension of $\dul_{1,2} $  is
\begin{equation}
\label{realdulac1}
\dul_{1,2} : ( 0,1) \rightarrow (0,  1) 
\end{equation}
which is a real-analytic involution
$$\dul_{1,2}(\bar{y})= \overline{\dul_{1,2}(y)}, \quad\dul_{1,2}^2=id, \quad \dul_{1,2}(p_c)=p_c .$$
For the needs of the present paper we do not need a global description of the Dulac maps, but only  an appropriate domain of analiticity in which we shall apply later the argument principle. This domain is described as follows.

 Let $ C_{\pm \pi}$, $C_{\pm 0}$ be the   real curves in the complex $y$-plane, defined in polar coordinates as follows
 $$
 C_{+0}= \{\rho = \frac{\sin(\varepsilon \varphi)}{\sin(\varphi + \varepsilon \varphi)} : - \frac{\pi}{1+\varepsilon} < \varphi <  \frac{\pi}{1+\varepsilon}  \}$$
 $$
 C_{-0}= \{\rho = \frac{\sin(\varepsilon \varphi)}{\sin(\varphi + \varepsilon \varphi)} : 0 < \varphi < \frac{\pi}{1+\varepsilon} \}$$
 $$
 C_{-\pi}= \{\rho = \frac{\sin(\varepsilon \varphi+ \varepsilon \pi)}{\sin(\varphi + \varepsilon \varphi+ \varepsilon \pi)} : 0 < \varphi < \frac{\pi(1-\varepsilon)}{1+\varepsilon} \} 
$$
$$
 C_{\pi}=
 \{\rho = \frac{\sin(\varepsilon \varphi- \varepsilon \pi)}{\sin(\varphi + \varepsilon \varphi- \varepsilon \pi)} : -\frac{\pi(1-\varepsilon)}{1+\varepsilon}< \varphi < 0 \}  .
 $$
 
Let ${\bf D}_1$ be the open complex domain delimited by the curves $C_{\pm 0}$ and $C_{\pm \pi}$,  and ${\bf D}_0$ be the complex domain delimited by the curves $C_{\pm 0}$ and the segment $(-\infty, 0)$, see fig. \ref{fig_isoclines}.
%\ref{fig_mapep}.
%The curves curves $C_\pm0$, $C_{\pm \pi}$ for one particular value of $\epsilon$ are shown on fig. \ref{fig_mapep}, see also  \
%fig. \ref{fig_isoclines}.
%\begin{figure}[htpb]
%\includegraphics[width=100mm]{figs/bgsf-mapep}
%\def\svgwidth{12cm}
%\input{figs/D0D1.pdf_tex}
%\caption{Numerical simulation of the curves $C_{\pm 0}$, $C_{\pm \pi}$ for $\varepsilon=0.0055$}
%\label{fig_mapep}
%\end{figure}
\begin{proposition}
\label{pr21}
%\label{dulacont}
The real Dulac maps (\ref{realdulac1})  allow an extension to   bi-holomorphic maps
$$
\dul_{1,2}: 
{\bf D}_0\cup {\bf D}_1 \cup C_{+0} \cup C_{-0}\cup \{y_c\}
\rightarrow {\bf D}_0\cup {\bf D}_1 \cup C_{+0} \cup C_{-0}\cup \{y_c\}
$$
where
$$
\dul_{1,2}({\bf D}_1)= {\bf D}_0,\; \dul_{1,2}( C_{+0}) = C_{-0} , \dul_{1,2}(y_c)=y_c .
$$
They can be further analytically continued to a suitable open neighborhood of 
$C_{\pi}$, $C_{-\pi}$
and 
$$
\dul_{1,2}(C_{-\pi}) = \dul_{1,2}(C_{\pi}) = (-\infty,0).
$$
The limit of $\dul_{1,2}$ at $y=1$ exists and $\dul_{1,2}(1)=0$.
\end{proposition}
{\bf Proof.} The function $H(0,y)=y^\varepsilon (1-y)$ allows an analytic continuation in $\bbC \setminus (-\infty,0]$.
The real Dulac map satisfies $H(0,y)=H(0,\dul_{1,2}(y))$ so does its complex extension, when it exists. Consider the 
isoclines 
$$
C_{\theta}=\{ y\in {\bf D}_0 \cup {\bf D}_1 : \arg ( y^\varepsilon (1-y) )= \varepsilon \theta \} $$
or equivalently
$$
C_{\theta} = \{ y :  \varepsilon \arg(y) + \arg (1-y) = \varepsilon \theta \}.
$$
which implies in  polar coordinates 
$$
C_\theta= \{ (\rho, \phi): \rho = \frac{\sin(\varepsilon \varphi - \theta)}{\sin(\varphi+\varepsilon \varphi - \theta)} \} .
$$
Thus ${\bf D}_1\cup{\bf D}_0$ is an union of the  isoclines $C_\theta$ 
$${\bf D}_1\cup{\bf D}_0 = \{  C_\theta : - \pi < \theta <   \pi \}
$$
where each  $C_\theta$ has exactly two connected components, contained respectively in ${\bf D}_1$ or ${\bf D}_0$,
see fig. \ref{fig_isoclines}.
\begin{figure}[htpb]
 \def\svgwidth{10cm}
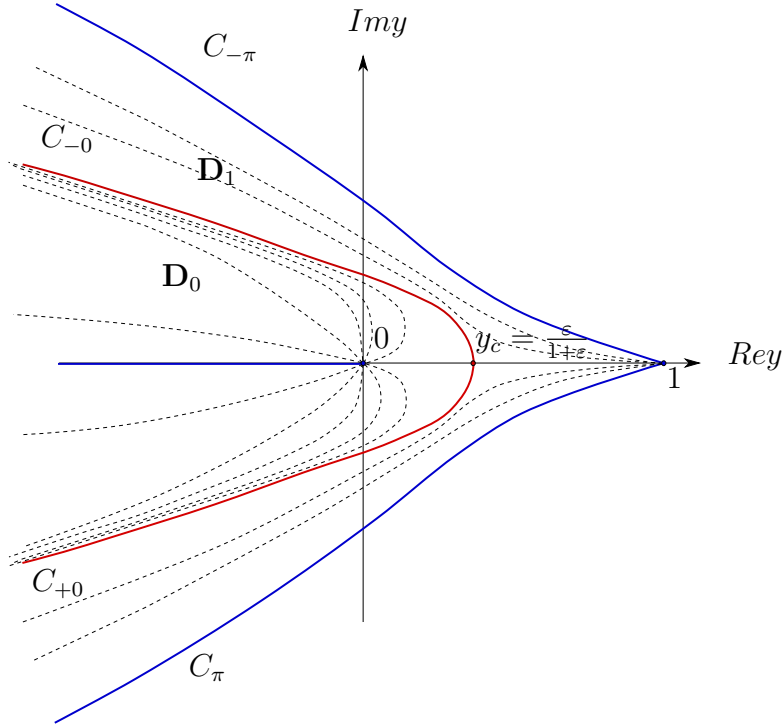
\caption{The isoclines of the function $y^\varepsilon (1-y)$}
\label{fig_isoclines}
\end{figure}
 For $\theta=0$ the Dulac map exchanges the segments $(0,y_c)$ and $(y_c,1)$, by continuity the same holds true for all $\theta$. Moreover, by a local analysis of the Dulac map near $1$ we have 
$$
\lim_{y\rightarrow 1} \dul_{1,2}(y)= 0
$$
and the identity $| y^\varepsilon (1-y) | = |\dul_{1,2}( y)^\varepsilon (1-\dul_{1,2}(y)) |$ implies that when $| y|$ tends to infinity along $C_\theta$, then so does $\dul_{1,2}(y)$. As $\dul_{1,2}'(y)\neq 0$ then $\dul_{1,2}$ is a bi-analytic map between connected components of $C_\theta$, a bijection between ${\bf D}_1 $ and ${\bf D}_0$, and finally a bi-holomorphic map between ${\bf D}_1 $ and ${\bf D}_0$. The claims about the behavior of $\dul_{1,2}$ along the border of the domains  ${\bf D}_1$, ${\bf D}_0$ are straightforward.$\Box$\\

Consider the continuous family of orbits $\{\gamma_y^{1,2}\}_{y}$, $y\in (\varepsilon/(1+\varepsilon),1)$ defining $\dul_{1,2}(y)$ on
fig. \ref{fig_fig21} respectively. Each $\gamma_y^{1,2}$ is a real path starting at $y$ and terminating at $\dul_{1,2}(y)$.
\begin{definition}
We shall say that the analytic functions $\dul_{1,2}$ defined in the domain ${\bf D}_1$ allow a geometric realization, provided that the continuous families of real paths $\{\gamma_y^{1,2}\}_{y}$, $y\in (y_c,1)$ allow an extension to  
continuous families of paths $\{\gamma_y\}_{y\in {\bf D}_1}$, contained in the complex leaves of the foliation 
${\mathcal F}_{\varepsilon,0}$, such that each $\gamma_y^{1,2}$ starts at $y$ and terminates at   
$\dul_{1,2}(y)$. 
\end{definition}
Of course, the families of paths $\{\gamma_y\}_{y}$ are defined up to a homotopy.
We note that although $\dul_1=\dul_2$ they allow non-equivalent geometric realizations.
\begin{proposition}
\label{pr23}
The real Dulac maps $\dul_{1,2} $ allow a geometric realization in the domain ${\bf D}_1$.
\end{proposition}
{\bf Proof.}
Consider the projection
$$
\pi : \bbC \rightarrow \bbC : (x,y) \mapsto y
$$
and let $\Gamma_y$ be the connected component of a leaf of the foliation ${\mathcal F}_{\varepsilon,0}$ through the point $(0,y) $ which is contained in the pre-image under $\pi$ of the domain $\overline{{\bf D}}_0 \cup {\bf D}_1$. If $c= y^\varepsilon (1-y)$, then 
along $\Gamma_y$ holds
\begin{equation}
\label{leaf}
x^2= y + c (1-y)^{-\frac{1}{\varepsilon}}
\end{equation}
which shows that $\pi: \Gamma_y \rightarrow \overline{{\bf D}}_0 \cup {\bf D}_1$ is a double covering, ramified at the points  $y$ and $\dul_1(y)=\dul_2(y)$. The leaf $\Gamma_y$ is therefore a smooth open Riemann surface, except $\Gamma_{y_c}$ which has a normal crossing at $(0,y_c)$, where $y_c= \dul_1(y_c)=\dul_2(y_c)$

Let $\sigma_y$ be a path in $\overline{{\bf D}}_0 \cup {\bf D}_1$ connecting $y$ to $\dul_1(y)=\dul_2(y)$. Such a path can be lifted in $\Gamma_y$ under $\pi$ in two different ways. To every continuous family of paths $\sigma_y$ correspond therefore two continuous families of lifts in $\gamma_y \subset \Gamma_y$. We apply now these considerations to the continuous family of segments $\sigma_y=[y,\dul_{1,2}(y)]\subset \bbR$, $y\in (y_c,1)$. The lift of $\sigma_y$ will be the real orbit $\gamma_y^1$ or $\gamma_y^2$ defining $\dul_1(y)$ or $\dul_2(y)$, and shown on fig. \ref{fig_fig21}. 

To complete the proof we
\begin{itemize}
\item 
extend first the segments $\sigma_y=[y,\dul(y)]\subset \bbR$ to a continuous family of paths $\sigma_y \subset  {\bf D}_1\cup {\bf D}_0 $, 
$y\in {\bf D}_1 $ 
which  connect $y\in {\bf D}_1 $ to $\dul (y)\in {\bf D}_0$, as it is illustrated on fig. \ref{fig_dulg}.
\\
\begin{figure}[htpb]
 \def\svgwidth{8cm}
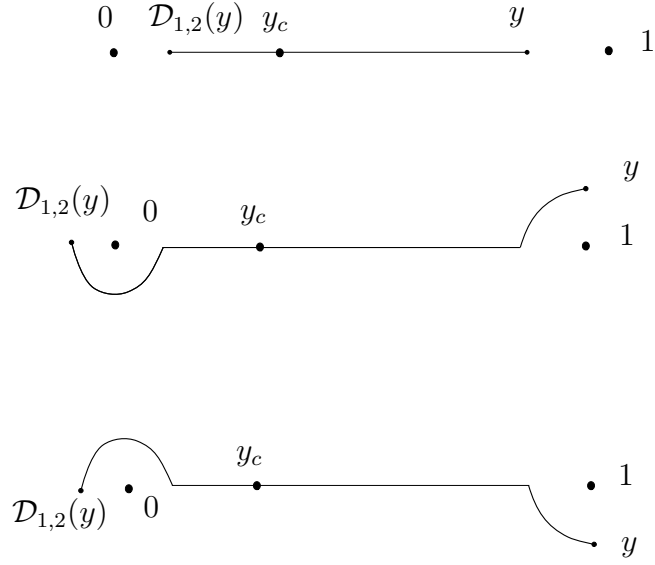
\caption{The continuous family of paths $\{ \sigma_y \}_y $}
\label{fig_dulg}
\end{figure}
\item
Second, to each path $\sigma_y$, $y\in {\bf D}_1 $, we associate its lift $\gamma_y$ with respect to $\pi$ with initial point $(0,y)$. For $y\in (y_c,1)$ the lift $\gamma_y$ was already defined, by continuity it will be defined without ambiguity for all $y\in {\bf D}_1$. For $y\in (y_c,1)$ the end point of $\gamma_y$ is
$(0, \dul(y))$. As the end point of $\gamma_y$ depends analytically on $y$, then it is 
$(0, \dul(y))$ for all $y\in {\bf D}_1$.

\end{itemize}
$\Box$
\begin{remark}
The above considerations show that the Riemann surface $\Gamma_y$ is a topological cylinder, which is a double covering of the topological disc $\overline{{\bf D}}_0 \cup {\bf D}_1$, ramified over $y$ and $\dul(y)$. The closed loop $\gamma_y^1 \circ( \gamma_y^2)^{-1}$ is the generator of the fundamental group of the cylinder.
\end{remark}
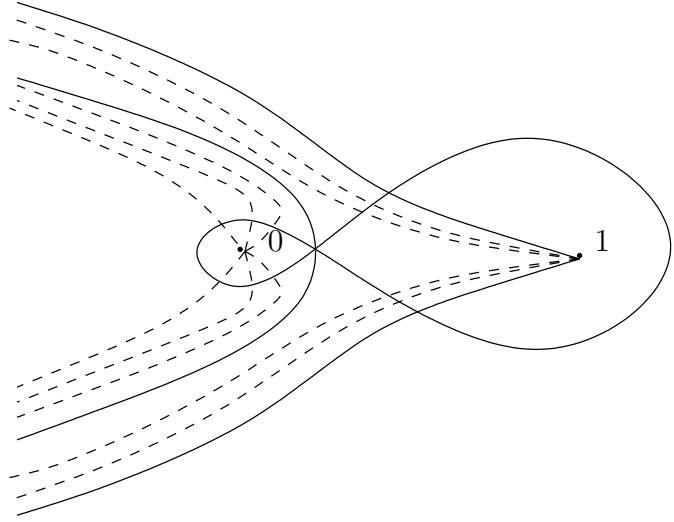
\begin{figure}[htpb]
\input{figs/csingc.pspdftex}
\caption{The curve $\{y: |f(y)|=|f(y_c)| \} .$}
\label{fig_csingc}
\end{figure}
\begin{remark}
It can be shown, by making use of Proposition \ref{pr22} that 
$\dul_{1,2}$ allows also an analytic extension to ${\mathbb C} \setminus \{1\}$ with countably many algebraic singularities. More precisely, if $y_0\neq 1$ is a singular point, then $y_0$ belongs to the leaf through the center point $p_c$. The curve on the $y$-plane containing the possible singular points is defined therefore by the equation 
$$\{y: |f(y)|=|f(y_c)| \} .$$
This curve is easily analyzed and it  is shown on fig.\ref{fig_csingc}.  
\end{remark}

\subsection{Geometric variation}
We describe a geometric construction of the variation. It will be used later, to control intersection points of curves $\{Im \dul_1=0\}$ and $\{Im \dul_2=0\}$. 
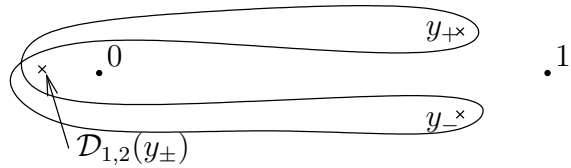
\begin{figure}[htpb]
\input{figs/y8fig.pspdftex}
\caption{Projection of the figure eight loop $\gamma_8$ to the $x=0$ transversal}
\label{fig_y8loop}
\end{figure}

\begin{figure}[htpb]
\input{figs/fig8.pspdftex}
\caption{The figure eight loop $\gamma_8$ at the limit $\ep\to 0$}
\label{fig_fig8}
\end{figure}

In the following proposition the dot denotes product in the path grupoid, i.e. $\gamma^1\cdot\, \gamma^2$ denotes path consisting of $\gamma^1$ followed by $\gamma^2$, where the end point of $\gamma_1$ conincide with the starting point of $\gamma_2$. The analog meaning for the inverse.
\begin{proposition}
\label{geo8}
Let $\gamma^{1,2}_y$ be a geometric realization of Dulac maps $\dul_{1,2}$ and let $y_\pm\in C_{\pm\pi}$ and $y_-=\overline{y_+}$.
The path $\gamma_8=(\gamma^1_{y_+}\cdot (\gamma^2_{y_-})^{-1})\cdot \overline{(\gamma^1_{y_+}\cdot (\gamma^2_{y_-})^{-1})}$ is a closed loop that projects on the $y$-plane to the loop shown on figure \ref{fig_y8loop}. It can be homotopically deformed to a loop located in finite ($\ep$-independent) distance from the slow parabola $y-x^2=0$. At the limit $\ep\to 0$ it is the figure eight loop in the leaf of fast foliation -- see figure \ref{fig_fig8}. 
\end{proposition}

\begin{proof}
By definition of the curves $C_{\pm\pi}$, $\dul_{1,2}(y_\pm)\in\bbR_-$; due to the conjugation relation $y_-=\overline{y_+}$ and the Schwartz reflection principle, the images of $y_\pm$ coincide. Thus, the path $(\gamma^1_{y_+}\cdot (\gamma^2_{y_-})^{-1})$ joins $Y_+$ with $y_-$ and passes through the third ramification point $\dul_{1,2}(y_\pm)$ of the double covering \ref{leaf}. It can be deformed homotopically to the path avoiding the ramification point $\dul_{1,2}(y_\pm)$ from the left. 

The next path is a complex conjugacy of the first one. It is the same composition of $\gamma$ paths with indices $(1,2)$ exchanged and the direction reversed. The composition is as shown on the figure \ref{fig_y8loop} after homotopical deformation that avoiding the ramification points $y_\pm$. 
\end{proof}

\section{Blowing up the turning point}
\begin{figure}[htpb]
 \def\svgwidth{6cm}
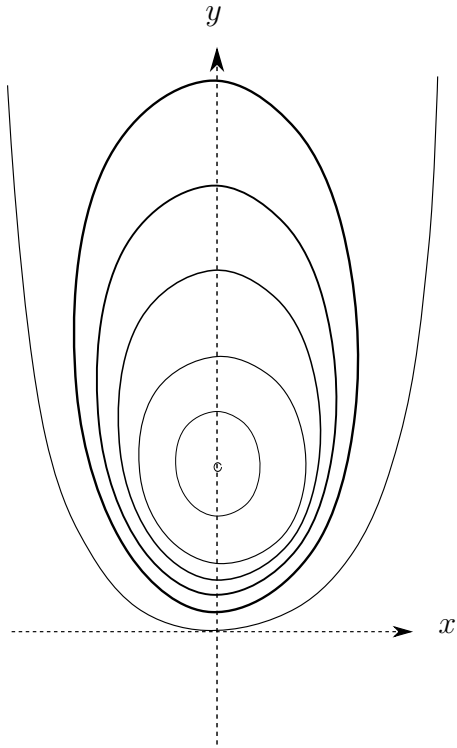
\caption{The real level sets of $ e^{-y } (y-x^2)$ }
\label{fig_scaledportarit}
\end{figure}
When $\varepsilon$ tends to zero, the center $p_c$ of the integrable foliation tends to the contact point $p_0=(0,0)$ of the slow manifold $\{ y=x^2 \}$ with the leaves $y=const$. The point $p_0$ is therefore the turning point of our slow-fast Darboux foliation ${\mathcal F}_{\varepsilon,0}$, and the study of the Dulac map near $p_0$ when $\varepsilon$ tends to zero will be studied, as explained in detail in \cite[section 3]{bmn13}, by a weighted blow up in the $(x,y,\varepsilon)$-space. Namely, the rescaling 
\begin{equation}
\label{scale}
x \rightarrow \sqrt{\varepsilon} x, \; y \rightarrow \varepsilon y, \; \varepsilon  \rightarrow \varepsilon
\end{equation}
sends
the center point $y_c$ to $1/(1+\varepsilon)$, leaves the parabola $y=x^2$ invariant and 
transforms the first integral $[(1-y)(y-x^2)^\varepsilon]^{1/\varepsilon}$ of the foliation 
${\mathcal F}_{\varepsilon,0}$ to
the form
$$
\varepsilon(1-\varepsilon y)^{\frac{1}{\varepsilon} }(y-x^2) = \varepsilon e^{-y + O(\varepsilon)} (y-x^2) = \varepsilon e^{-y } (y-x^2) + O(\varepsilon^2).
$$
The blown up foliation has therefore, in every compact neighborhood of the origin, an analytic first integral, uniformly in $x,y$, $O(\varepsilon)$-close to 
$$
e^{-y } (y-x^2) 
$$
see fig. \ref{fig_scaledportarit}.
Finally, the curves $C_{\pm \pi}$ are transformed to curves on  a finite distance from the origin $y=0$, see fig. \ref{fig_scaledisoclines}.
\begin{figure}[htpb]
 \def\svgwidth{6cm}
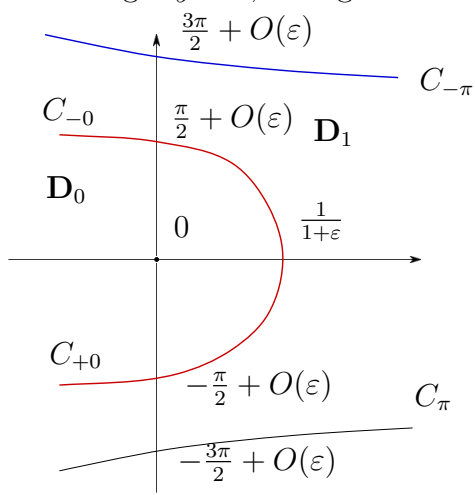
\caption{The cross-section $\{x=0\}$ in a neighborhood of the turning point  after the weighted blow up of the $(x,y,\varepsilon)$-space}
\label{fig_scaledisoclines}
\end{figure}

\section{The perturbed foliation ${\mathcal F}_{\varepsilon,\delta}$ where $\varepsilon >0$ and $\|\delta\|$ is much smaller than $\varepsilon$ }

\begin{figure}[htpb]
\input{figs/dulacmaps.pspdftex}
\caption{Two Dulac maps for the perturbed system}
\label{fig_dul}
\end{figure}
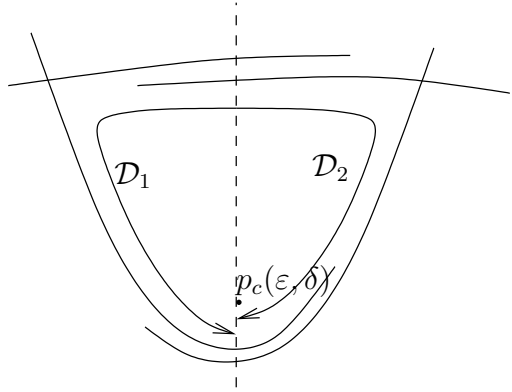

In the preceding section we considered two Dulac maps $\dul_1$ and $\dul_2$ which coincide as analytic functions, but were defined geometrically in two different ways. We have respectively two non-equivalent geometric realizations $\{\gamma_y^{1,2}\}_{y\in {\bf D}_1}$ of the same analytic function $\dul = \dul_1 = \dul_2$.
In this section we suppose that $\varepsilon>0$ is fixed, $\delta\in \bbR^N$, and $\|\delta\|$ is sufficiently small with respect to $\varepsilon$. Let $p_c=(x_c,y_c)$ be the singular point (a focus) of ${\mathcal F}_{\varepsilon,\delta}$, close to the center of ${\mathcal F}_{\varepsilon,0}$. We consider the cross-section $\{ x= x_c \}$ and as before define geometric realizations of the corresponding Dulac maps $\dul_1$ and $\dul_2$. Indeed, when $\delta \neq 0$, the   continuous family of paths $\{\gamma_y^{1,2}\}_{y\in {\bf D}_1}$, contained in the leaves of the foliation 
${\mathcal F}_{\varepsilon,0}$ persist under  a small perturbation, at least when the initial point $(0,y)$ belongs to a relatively compact domain $K\cup {\bf D}_1$, and the leaves of the foliation ${\mathcal F}_{\varepsilon,0}$ are transverse to the line $\{x=0\}$ at $(0,y)$, as well at $(0,\dul(y))$. We obtain in this way two continuous families of paths in the leaves of the perturbed foliation ${\mathcal F}_{\varepsilon,\delta}$, which begin at a point $(0,y)$ and terminate at a point $(0, \dul^{1}(y)$ or $(0, \dul^{2}(y)$ respectively. 

The above holds true also in a neighborhood of the center point $p_c$ (which becomes a focus after the perturbation), but  also in a neighborhood of the singular point $(0,1)\in \{ x=0 \}$. Indeed, after the perturbation a saddle point persists, and the Dulac map near such a point is defined in every fixed sector, centered at the singular point. The image of the real analytic curves $C_{\pm \pi}$ under the Dulac map is the
 negative real semi-axes, see Proposition \ref{pr21}. Therefore these curve can be also defined by the condition that 
 $\{y : Im (\dul^{1}(y)) = 0$, and $Im (\dul^{2}(y)) = 0$. We denote for a further use these two curves by $C_{\pm \pi}^1=C_{\pm \pi}^1 (\varepsilon, \delta) $ and $C_{\pm \pi}^2=C_{\pm \pi}^2 (\varepsilon, \delta) $  respectively   (when there is no ambiguity,  the dependence on $\varepsilon$ and $\delta$ is omitted).

 As shown in \cite{gavr11,gavr11a}, the curves $C_{\pm \pi}^{1,2}$ are analytic, including at the singular points of the Dulac maps, corresponding to the saddles $s_1$ and $s_2$.  
 \begin{figure}[htpb]
 \def\svgwidth{10cm}
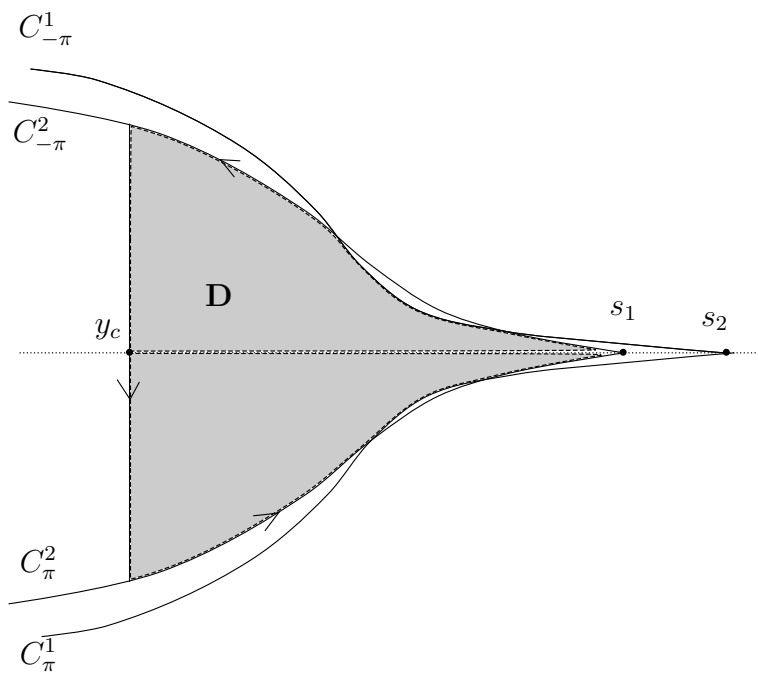
\caption{The complex domain ${\mathbf D}$,  bounded by $C_{\pm \pi}^{1,2}$,  and the line ${\mathcal Re}\; y > y_c$}
\label{fig_domain}
\end{figure}
Consider the closed complex domain ${\mathbf D}$, which is bounded by $C_{\pm \pi}^{1,2}$,  and the line ${\mathcal Re}\; y > y_c$, as on fig.\ref{fig_domain}.

\begin{proposition}
\label{argument}
The number of the  limit cycles of ${\mathcal F}_{\varepsilon,\delta}$, bifurcating from the closed period annulus of ${\mathcal F}_{\varepsilon,0}$, is bounded by the variation of the argument of the analytic function $\dul_1 - \dul_2$ divided by $2\pi$, along the border $ \partial  \mathbf D$ of the domain shown on fig. \ref{fig_domain}.
\end{proposition}
{\bf Proof.} The Dulac maps have   analytic continuations in $\mathbf D$, as their geometric realizations persist under sufficiently small perturbation. The limit cycles are in a one-to-one correspondence with the fixed points $y$, $\dul_1(y)=\dul_2(y)$ of the first return map. Therefore it is enough to bound the zeros of $\dul_1-\dul_2$ by making use of the argument principle, as explained  in \cite[section 2.1]{gavr11a}.$\Box$

\begin{corollary} 
For every fixed $\varepsilon>0$ and for every sufficiently small $\delta$, the Dulac maps $\dul_1, \dul_2$ are real analytic functions in a suitable neighborhood of ${\mathbf D} \setminus \{s_i\}$ where $s_i\sim 1$ are the singular points of the Dulac maps. 
\end{corollary}

Our intension is to apply the argument principle to the analytic function $\dul_1- \dul_2$ in ${\mathbf D}$, in order to bound its complex zeros (equivalently, fixed points of the return map, or complex limit cycles). For this we note that although the Dulac map is singular at $s_i$, it is continuous at these points. We shall bound uniformly one one side, the variation of the argument of $\dul_1- \dul_2$ along the segment $\{{\mathcal Re}\; y > y_c\} \cap {\mathbf D}$, and on the other hand the number of the zeros of the imaginary part of $\dul_1- \dul_2$ along $\{C_1^\pm \cup C_2^\pm\} \cap  {\mathbf D}$. By definition of the curves $C_1^\pm , C_2^\pm$, these zeros are just the intersection points $C_1^+ \cap C_2^+$+ and $C_1^- \cap C_2^-$, counted with multiplicity. The curves $C_1^\pm,  C_2^\pm$ depend, however, on $\varepsilon$, $\delta$, and their behavior when the parameters $\varepsilon, \delta$ tend to zero is crucial. When $\delta=0$, $\varepsilon>0$ the curves are explicit and tend to the real axes as $\varepsilon \rightarrow 0$. 

\section{Proof of Theorem \ref{th:main}}
 \begin{figure}[htpb]
 \def\svgwidth{12cm}
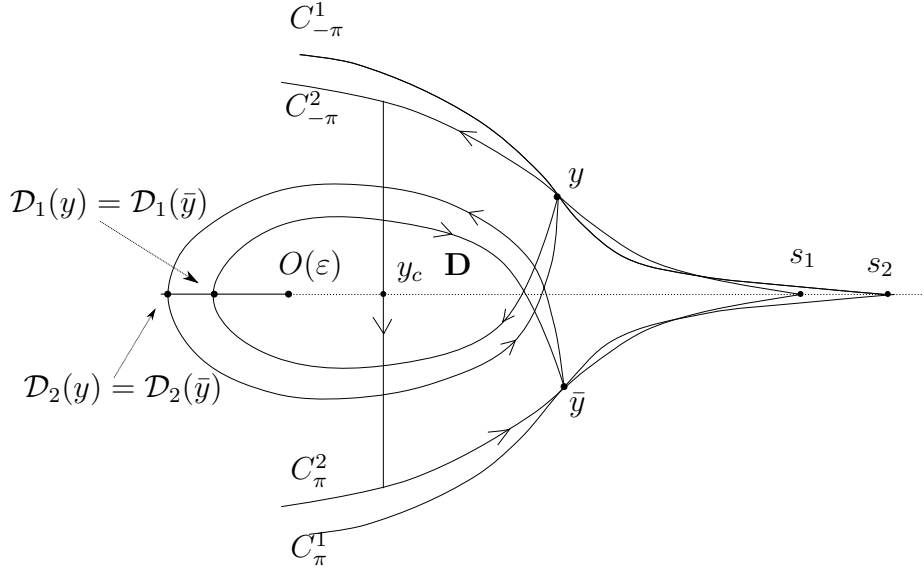
\caption{The geometric realization of the holonomy ${\mathcal Hol}$ and its projection on the cross-section.}
\label{fig_domain1}
\end{figure}
%\begin{figure}[htpb]
%\def\svgwidth{12cm}
%\input{figs/eightloop.pdf_tex}
%\caption{The figure eight loop of the doubly punctured plane $\{y=c\}$}
%\label{fig_eightloop}
%\end{figure}
Suppose first that $\varepsilon>0$ belongs to a sufficiently small but fixed neighborhood of the origin.
Consider the segment corresponding to the part of $ \partial  \mathbf D$, contained in the line $\{ {\mathcal Re}\; y = y_c \}$. It follows from Proposition \ref{pr23} that along this segment the
functions $\dul_1$, $ \dul_2$ have a geometric realization and hence 
are analytic, both in $y$ and $\varepsilon$, $\delta$, provided that $\varepsilon \neq 0$ and $\|\delta\|$ is much smaller than $\varepsilon$. To prove the analyticity in a neighborhood  of $\varepsilon = 0$ we consider the rescaling (\ref{scale}) and the domain   ${\bf D}_1$ shown on fig.\ref{fig_scaledisoclines}. It can be shown along the same lines as in Proposition \ref{pr23}, that the Dulac maps have a geometric realization in ${\mathbf D}_1$ again, and hence it is also analytic 
in $\sqrt{\varepsilon}, \delta$, in 
 a neighborhood of $\varepsilon=0, \delta=0$. 
This implies the analyticity for $\varepsilon$, $\delta$ close to zero (after the rescaling).
Therefore the variation of the argument along the compact segment, corresponding to the part of $ \partial  \mathbf D$, contained in the line $\{ {\mathcal Re}\; y = y_c \}$, is uniformly bounded.

Along the remaining part of $ \partial  \mathbf D$ the displacement map $\dul_1- \dul_2$ is analytic, except at the singular points $s_i$. For this purpose we bound the increase of the argument of $\dul_1- \dul_2$ by the number of the zeros of its imaginary part (the so called "Petrov trick" ) along $ \partial  \mathbf D$. Clearly, these zeros are exactly the intersection points of the curves $\{Im \; \dul_1 = 0 \}$, $\{Im \; \dul_2 = 0 \}$, that is to say $C^1_{-\pi} \cap C^2_{-\pi}$ and  $C^1_{+\pi}\cap C^2_{+\pi}$. The intersection points have a transparent geometric meaning: they correspond to complex limit cycles intersecting the cross-section, or fixed points of the holonomy map ${\mathcal Hol}$ along the "figure height loop" which we recall now (see \cite{gavr11a,gavr11, bmn13}). If the holonomy map ${\mathcal Hol}$ were analytic with respect to  the parameters too, this would imply  an uniform bound for the number of fixed points of $Z(\varepsilon)$, 
 when $\varepsilon>0$ belongs to a sufficiently small neighborhood of the origin.

Indeed, the one-dimensional leaves of the foliation ${\mathcal F}_{0,0}$ are the punctured discs $ \{ (x,y): y=c\} \setminus \{ ( \pm \sqrt{c},c) \}$, where $c\neq 0$. The geometric realization of ${\mathcal Hol}=id$ is then explained in proposition \ref{geo8} and shown on figures \ref{fig_y8loop} ,\ref{fig_fig8} from which the analyticity follows, except along the leaf $ \{  y=0 \}$ through the turning point $(0,0)$. In a neighborhood of the turning point   we use the rescaling (\ref{scale}) and describe the geometric realization of ${\mathcal Hol}$ as follows. 
Let $y\in C_{-\pi}$ on fig. \ref{fig_domain1} and note that $\bar{y} \in C_{+\pi}$. Consider the paths in the leaves of the foliation ${\mathcal F}_{0,0}$
$$
\gamma_y^1,( \gamma_{\bar{y}}^1)^{-1}, \gamma_y^2, (\gamma_{\bar{y}}^2)^{-1} .
$$
connecting the points $y, \dul_1(y), \bar{y}, \dul_2(y), y$.
These paths can be composed and the resulting closed path defines the holonomy map ${\mathcal Hol}$ which is then obviously analytic in $\varepsilon, \delta$. The projection of these four paths on the cross-section, in the case of a fixed point of ${\mathcal Hol}$ are shown on  fig. \ref{fig_domain1}.

In a completely similar way we may prove the uniform boundedness of $Z(\varepsilon)$ for $\varepsilon>0$ in a neighborhood of $\varepsilon = \infty$. Suffice it to exchange the roles of $P_0$ and $P_1$ , and consider the foliation with a first integral $P_0 P_1 ^{1/\varepsilon}$, where $1/\varepsilon>0$ belongs  to a sufficiently small neighborhood of the origin.

Finally, when $\varepsilon$ belongs to a compact subset of $(0,\infty)$, the uniform boundedness of $Z(\varepsilon)$ follows from \cite{gavr11a}. The Theorem is proved.

\bibliographystyle{plain}
\bibliography{../BIBfiles/bibliography}
%\bibliography{../bibliography}
%\bibliography{bibliography}
\end{document}

%% file: figs/slowfast.pdf_tex
%% Creator: Inkscape inkscape 0.48.1, www.inkscape.org
%% Accompanies image file 'slowfast.pdf' (pdf, eps, ps)

\begingroup
  \makeatletter
  \providecommand\color[2][]{%
    \errmessage{(Inkscape) Color is used for the text in Inkscape, but the package 'color.sty' is not loaded}
    \renewcommand\color[2][]{}%
  }
  \providecommand\transparent[1]{%
    \errmessage{(Inkscape) Transparency is used (non-zero) for the text in Inkscape, but the package 'transparent.sty' is not loaded}
    \renewcommand\transparent[1]{}%
  }
  \providecommand\rotatebox[2]{#2}
  \ifx\svgwidth\undefined
    \setlength{\unitlength}{546.74858398pt}
  \else
    \setlength{\unitlength}{\svgwidth}
  \fi
  \global\let\svgwidth\undefined
  \makeatother
  \begin{picture}(1,0.53201213)%
    \put(0,0){\includegraphics[width=\unitlength]{slowfast.pdf}}%
    \put(0.190947,0.00365228){\color[rgb]{0,0,0}\makebox(0,0)[lb]{\smash{$\varepsilon=0$}}}%
    \put(0.71769733,0.00365228){\color[rgb]{0,0,0}\makebox(0,0)[lb]{\smash{$\varepsilon > 0$}}}%
  \end{picture}%
\endgroup

%% file: figs/fig21.pdf_tex
%% Creator: Inkscape inkscape 0.48.1, www.inkscape.org
%% Accompanies image file 'fig21.pdf' (pdf, eps, ps)

\begingroup
  \makeatletter
  \providecommand\color[2][]{%
    \errmessage{(Inkscape) Color is used for the text in Inkscape, but the package 'color.sty' is not loaded}
    \renewcommand\color[2][]{}%
  }
  \providecommand\transparent[1]{%
    \errmessage{(Inkscape) Transparency is used (non-zero) for the text in Inkscape, but the package 'transparent.sty' is not loaded}
    \renewcommand\transparent[1]{}%
  }
  \providecommand\rotatebox[2]{#2}
  \ifx\svgwidth\undefined
    \setlength{\unitlength}{343.79221191pt}
  \else
    \setlength{\unitlength}{\svgwidth}
  \fi
  \global\let\svgwidth\undefined
  \makeatother
  \begin{picture}(1,1.0979009)%
    \put(0,0){\includegraphics[width=\unitlength]{fig21.pdf}}%
    \put(0.29311787,0.65225637){\color[rgb]{0,0,0}\makebox(0,0)[lb]{\smash{$\mathcal{D}_1$}}}%
    \put(0.59258806,0.65871971){\color[rgb]{0,0,0}\makebox(0,0)[lb]{\smash{$\mathcal{D}_2$}}}%
    \put(0.55380774,0.27953446){\color[rgb]{0,0,0}\makebox(0,0)[lb]{\smash{$p_c$}}}%
    \put(0.4827105,1.07668533){\color[rgb]{0,0,0}\makebox(0,0)[lb]{\smash{$\{x=0\}$}}}%
    \put(0.15523235,0.86123913){\color[rgb]{0,0,0}\makebox(0,0)[lb]{\smash{$s_1$}}}%
    \put(0.91575739,0.87847483){\color[rgb]{0,0,0}\makebox(0,0)[lb]{\smash{$s_2$}}}%
  \end{picture}%
\endgroup

%% file: figs/leafcov.pspdftex
\begin{picture}(0,0)%
\includegraphics{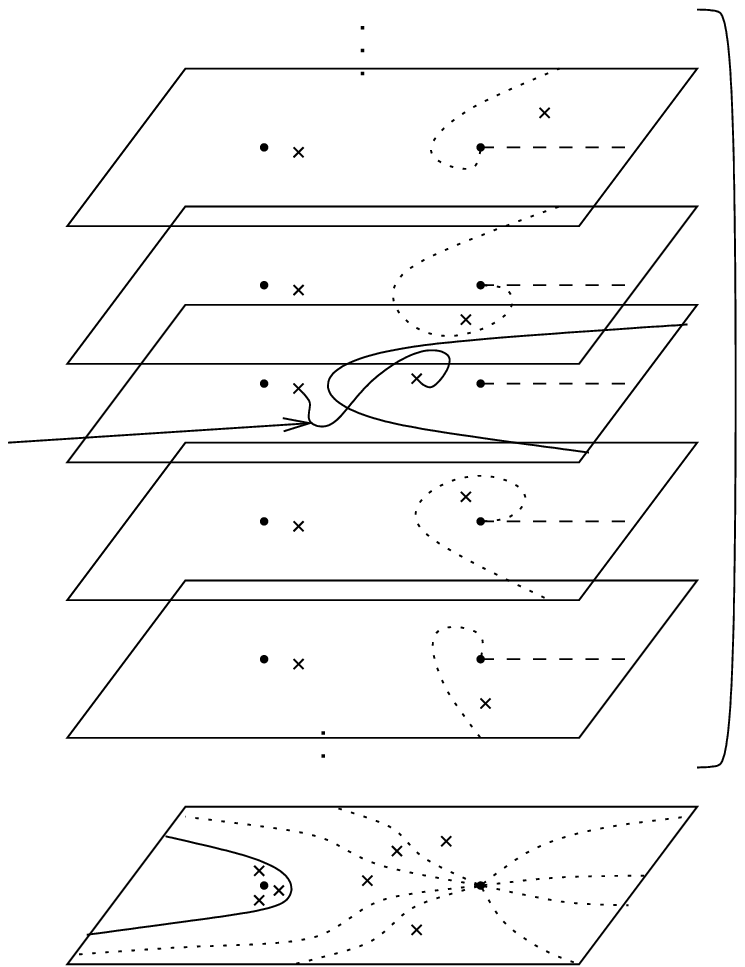}%
\end{picture}%
\setlength{\unitlength}{4144sp}%
\begingroup\makeatletter\ifx\SetFigFont\undefined%
\gdef\SetFigFont#1#2#3#4#5{%
  \reset@font\fontsize{#1}{#2pt}%
  \fontfamily{#3}\fontseries{#4}\fontshape{#5}%
  \selectfont}%
\fi\endgroup%
\begin{picture}(3765,4399)(6961,-4798)
\put(10711,-2176){\makebox(0,0)[lb]{\smash{{\SetFigFont{12}{14.4}{\rmdefault}{\mddefault}{\updefault}{\color[rgb]{0,0,0}$C_\ep$}%
}}}}
\put(9001,-2266){\makebox(0,0)[lb]{\smash{{\SetFigFont{12}{14.4}{\rmdefault}{\mddefault}{\updefault}{\color[rgb]{0,0,0}$y$}%
}}}}
\put(6976,-2356){\makebox(0,0)[lb]{\smash{{\SetFigFont{12}{14.4}{\rmdefault}{\mddefault}{\updefault}{\color[rgb]{0,0,0}$\gamma_y$}%
}}}}
\end{picture}%

%% file: figs/isoclines.pdf_tex
%% Creator: Inkscape inkscape 0.48.1, www.inkscape.org
%% Accompanies image file 'isoclines.pdf' (pdf, eps, ps)

\begingroup
  \makeatletter
  \providecommand\color[2][]{%
    \errmessage{(Inkscape) Color is used for the text in Inkscape, but the package 'color.sty' is not loaded}
    \renewcommand\color[2][]{}%
  }
  \providecommand\transparent[1]{%
    \errmessage{(Inkscape) Transparency is used (non-zero) for the text in Inkscape, but the package 'transparent.sty' is not loaded}
    \renewcommand\transparent[1]{}%
  }
  \providecommand\rotatebox[2]{#2}
  \ifx\svgwidth\undefined
    \setlength{\unitlength}{607.78198242pt}
  \else
    \setlength{\unitlength}{\svgwidth}
  \fi
  \global\let\svgwidth\undefined
  \makeatother
  \begin{picture}(1,0.95572962)%
    \put(0,0){\includegraphics[width=\unitlength]{isoclines.pdf}}%
    \put(0.48130337,0.4976865){\color[rgb]{0,0,0}\makebox(0,0)[lb]{\smash{$0$}}}%
    \put(0.61472423,0.49929399){\color[rgb]{0,0,0}\makebox(0,0)[lb]{\smash{$y_c=\frac{\varepsilon}{1+\varepsilon}$}}}%
    \put(0.86629494,0.44463961){\color[rgb]{0,0,0}\makebox(0,0)[lb]{\smash{$1$}}}%
    \put(0.24835572,0.72080931){\color[rgb]{0,0,0}\makebox(0,0)[lb]{\smash{${\mathbf D}_1$}}}%
    \put(0.20254252,0.58015477){\color[rgb]{0,0,0}\makebox(0,0)[lb]{\smash{${\mathbf D}_0$}}}%
    \put(0.25639311,0.8847723){\color[rgb]{0,0,0}\makebox(0,0)[lb]{\smash{$C_{-\pi}$}}}%
    \put(0.23790709,0.0633499){\color[rgb]{0,0,0}\makebox(0,0)[lb]{\smash{$C_{\pi}$}}}%
    \put(0.03128301,0.17616309){\color[rgb]{0,0,0}\makebox(0,0)[lb]{\smash{$C_{+0}$}}}%
    \put(0.04304945,0.76502483){\color[rgb]{0,0,0}\makebox(0,0)[lb]{\smash{$C_{-0}$}}}%
    \put(0.44670748,0.9190796){\color[rgb]{0,0,0}\makebox(0,0)[lb]{\smash{$Im y$}}}%
    \put(0.94683795,0.47691879){\color[rgb]{0,0,0}\makebox(0,0)[lb]{\smash{$Re y$}}}%
  \end{picture}%
\endgroup

%% file: figs/fig_duls.pdf_tex
%% Creator: Inkscape inkscape 0.48.1, www.inkscape.org
%% Accompanies image file 'fig:duls.pdf' (pdf, eps, ps)

\begingroup
  \makeatletter
  \providecommand\color[2][]{%
    \errmessage{(Inkscape) Color is used for the text in Inkscape, but the package 'color.sty' is not loaded}
    \renewcommand\color[2][]{}%
  }
  \providecommand\transparent[1]{%
    \errmessage{(Inkscape) Transparency is used (non-zero) for the text in Inkscape, but the package 'transparent.sty' is not loaded}
    \renewcommand\transparent[1]{}%
  }
  \providecommand\rotatebox[2]{#2}
  \ifx\svgwidth\undefined
    \setlength{\unitlength}{443.81733398pt}
  \else
    \setlength{\unitlength}{\svgwidth}
  \fi
  \global\let\svgwidth\undefined
  \makeatother
  \begin{picture}(1,0.94554788)%
    \put(0,0){\includegraphics[width=\unitlength]{fig_duls.pdf}}%
    \put(0.74474501,0.88288306){\color[rgb]{0,0,0}\makebox(0,0)[lb]{\smash{$y$}}}%
    \put(0.1523069,0.87432181){\color[rgb]{0,0,0}\makebox(0,0)[lb]{\smash{$\mathcal{D}_{1,2}(y)$}}}%
    \put(0.9604884,0.83322784){\color[rgb]{0,0,0}\makebox(0,0)[lb]{\smash{$1$}}}%
    \put(0.06840671,0.87260955){\color[rgb]{0,0,0}\makebox(0,0)[lb]{\smash{$0$}}}%
    \put(0.33894209,0.87432176){\color[rgb]{0,0,0}\makebox(0,0)[lb]{\smash{$y_c$}}}%
    \put(-0.06686096,0.5712537){\color[rgb]{0,0,0}\makebox(0,0)[lb]{\smash{$\mathcal{D}_{1,2}(y)$}}}%
    \put(-0.07199771,0.05244229){\color[rgb]{0,0,0}\makebox(0,0)[lb]{\smash{$\mathcal{D}_{1,2}(y)$}}}%
    \put(0.93309247,0.62775792){\color[rgb]{0,0,0}\makebox(0,0)[lb]{\smash{$y$}}}%
    \put(0.92966793,0.0044993){\color[rgb]{0,0,0}\makebox(0,0)[lb]{\smash{$y$}}}%
    \put(0.30298485,0.55584346){\color[rgb]{0,0,0}\makebox(0,0)[lb]{\smash{$y_c$}}}%
    \put(0.29613587,0.15517725){\color[rgb]{0,0,0}\makebox(0,0)[lb]{\smash{$y_c$}}}%
    \put(0.14211665,0.55134414){\color[rgb]{0,0,0}\makebox(0,0)[lb]{\smash{$0$}}}%
    \put(0.14382893,0.06164094){\color[rgb]{0,0,0}\makebox(0,0)[lb]{\smash{$0$}}}%
    \put(0.92632665,0.51196245){\color[rgb]{0,0,0}\makebox(0,0)[lb]{\smash{$1$}}}%
    \put(0.92461443,0.11472065){\color[rgb]{0,0,0}\makebox(0,0)[lb]{\smash{$1$}}}%
  \end{picture}%
\endgroup

%% file: figs/csingc.pspdftex
\begin{picture}(0,0)%
\includegraphics{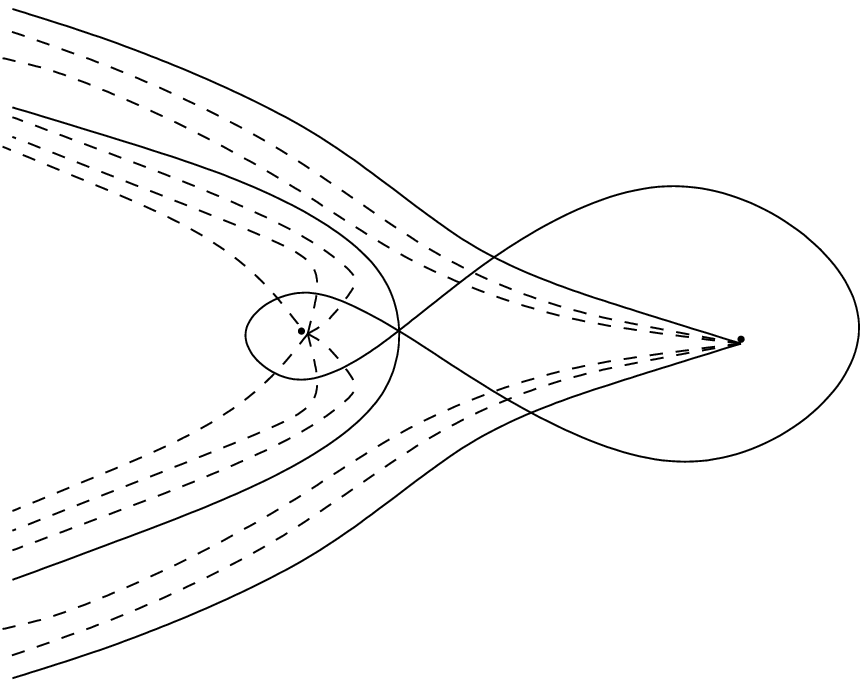}%
\end{picture}%
\setlength{\unitlength}{4144sp}%
\begingroup\makeatletter\ifx\SetFigFont\undefined%
\gdef\SetFigFont#1#2#3#4#5{%
  \reset@font\fontsize{#1}{#2pt}%
  \fontfamily{#3}\fontseries{#4}\fontshape{#5}%
  \selectfont}%
\fi\endgroup%
\begin{picture}(3939,3084)(6469,-4258)
\put(8011,-2671){\makebox(0,0)[lb]{\smash{{\SetFigFont{12}{14.4}{\rmdefault}{\mddefault}{\updefault}{\color[rgb]{0,0,0}$0$}%
}}}}
\put(9946,-2671){\makebox(0,0)[lb]{\smash{{\SetFigFont{12}{14.4}{\rmdefault}{\mddefault}{\updefault}{\color[rgb]{0,0,0}$1$}%
}}}}
\end{picture}%

%% file: figs/y8fig.pspdftex
\begin{picture}(0,0)%
\includegraphics{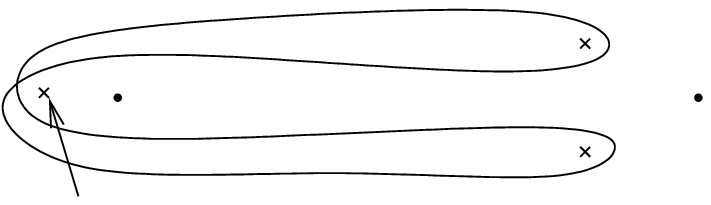}%
\end{picture}%
\setlength{\unitlength}{4144sp}%
\begingroup\makeatletter\ifx\SetFigFont\undefined%
\gdef\SetFigFont#1#2#3#4#5{%
  \reset@font\fontsize{#1}{#2pt}%
  \fontfamily{#3}\fontseries{#4}\fontshape{#5}%
  \selectfont}%
\fi\endgroup%
\begin{picture}(3253,982)(4548,-5984)
\put(5131,-5371){\makebox(0,0)[lb]{\smash{{\SetFigFont{12}{14.4}{\rmdefault}{\mddefault}{\updefault}{\color[rgb]{0,0,0}$0$}%
}}}}
\put(7786,-5371){\makebox(0,0)[lb]{\smash{{\SetFigFont{12}{14.4}{\rmdefault}{\mddefault}{\updefault}{\color[rgb]{0,0,0}$1$}%
}}}}
\put(7021,-5191){\makebox(0,0)[lb]{\smash{{\SetFigFont{12}{14.4}{\rmdefault}{\mddefault}{\updefault}{\color[rgb]{0,0,0}$y_+$}%
}}}}
\put(7021,-5731){\makebox(0,0)[lb]{\smash{{\SetFigFont{12}{14.4}{\rmdefault}{\mddefault}{\updefault}{\color[rgb]{0,0,0}$y_-$}%
}}}}
\put(4951,-5911){\makebox(0,0)[lb]{\smash{{\SetFigFont{12}{14.4}{\rmdefault}{\mddefault}{\updefault}{\color[rgb]{0,0,0}$\dul_{1,2}(y_\pm)$}%
}}}}
\end{picture}%

%% file: figs/fig8.pspdftex
\begin{picture}(0,0)%
\includegraphics{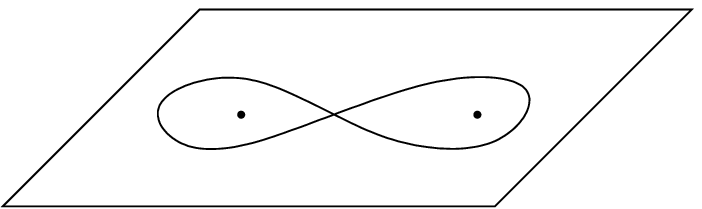}%
\end{picture}%
\setlength{\unitlength}{4144sp}%
\begingroup\makeatletter\ifx\SetFigFont\undefined%
\gdef\SetFigFont#1#2#3#4#5{%
  \reset@font\fontsize{#1}{#2pt}%
  \fontfamily{#3}\fontseries{#4}\fontshape{#5}%
  \selectfont}%
\fi\endgroup%
\begin{picture}(3174,924)(-11,-523)
\put(811,-61){\makebox(0,0)[lb]{\smash{{\SetFigFont{12}{14.4}{\rmdefault}{\mddefault}{\updefault}{\color[rgb]{0,0,0}$-\sqrt{c}$}%
}}}}
\put(1801,-151){\makebox(0,0)[lb]{\smash{{\SetFigFont{12}{14.4}{\rmdefault}{\mddefault}{\updefault}{\color[rgb]{0,0,0}$\sqrt{c}$}%
}}}}
\put(2656,209){\makebox(0,0)[lb]{\smash{{\SetFigFont{12}{14.4}{\rmdefault}{\mddefault}{\updefault}{\color[rgb]{0,0,0}$y=c$}%
}}}}
\end{picture}%

%% file: figs/scaledportarit.pdf_tex
%% Creator: Inkscape inkscape 0.48.1, www.inkscape.org
%% Accompanies image file 'scaledportarit.pdf' (pdf, eps, ps)

\begingroup
  \makeatletter
  \providecommand\color[2][]{%
    \errmessage{(Inkscape) Color is used for the text in Inkscape, but the package 'color.sty' is not loaded}
    \renewcommand\color[2][]{}%
  }
  \providecommand\transparent[1]{%
    \errmessage{(Inkscape) Transparency is used (non-zero) for the text in Inkscape, but the package 'transparent.sty' is not loaded}
    \renewcommand\transparent[1]{}%
  }
  \providecommand\rotatebox[2]{#2}
  \ifx\svgwidth\undefined
    \setlength{\unitlength}{319.69575195pt}
  \else
    \setlength{\unitlength}{\svgwidth}
  \fi
  \global\let\svgwidth\undefined
  \makeatother
  \begin{picture}(1,1.62579366)%
    \put(0,0){\includegraphics[width=\unitlength]{scaledportarit.pdf}}%
    \put(0.94649699,0.2470886){\color[rgb]{0,0,0}\makebox(0,0)[lb]{\smash{$x$}}}%
    \put(0.43558176,1.602979){\color[rgb]{0,0,0}\makebox(0,0)[lb]{\smash{$y$}}}%
  \end{picture}%
\endgroup

%% file: figs/scaledisoclines.pdf_tex
%% Creator: Inkscape inkscape 0.48.1, www.inkscape.org
%% Accompanies image file 'scaledisoclines.pdf' (pdf, eps, ps)

\begingroup
  \makeatletter
  \providecommand\color[2][]{%
    \errmessage{(Inkscape) Color is used for the text in Inkscape, but the package 'color.sty' is not loaded}
    \renewcommand\color[2][]{}%
  }
  \providecommand\transparent[1]{%
    \errmessage{(Inkscape) Transparency is used (non-zero) for the text in Inkscape, but the package 'transparent.sty' is not loaded}
    \renewcommand\transparent[1]{}%
  }
  \providecommand\rotatebox[2]{#2}
  \ifx\svgwidth\undefined
    \setlength{\unitlength}{506.26489258pt}
  \else
    \setlength{\unitlength}{\svgwidth}
  \fi
  \global\let\svgwidth\undefined
  \makeatother
  \begin{picture}(1,1.01581416)%
    \put(0,0){\includegraphics[width=\unitlength]{scaledisoclines.pdf}}%
    \put(0.63421123,0.5834993){\color[rgb]{0,0,0}\makebox(0,0)[lb]{\smash{$\frac{1}{1+\varepsilon}$}}}%
    \put(0.36143458,0.56426506){\color[rgb]{0,0,0}\makebox(0,0)[lb]{\smash{$0$}}}%
    \put(0.35618884,0.80731609){\color[rgb]{0,0,0}\makebox(0,0)[lb]{\smash{$\frac{\pi}{2}+O(\varepsilon)$}}}%
    \put(0.38591453,0.21804854){\color[rgb]{0,0,0}\makebox(0,0)[lb]{\smash{$-\frac{\pi}{2}+O(\varepsilon)$}}}%
    \put(0.37367453,1.00140717){\color[rgb]{0,0,0}\makebox(0,0)[lb]{\smash{$\frac{3 \pi}{2}+O(\varepsilon)$}}}%
    \put(0.37017738,0.0501859){\color[rgb]{0,0,0}\makebox(0,0)[lb]{\smash{$-\frac{3\pi}{2}+O(\varepsilon)$}}}%
    \put(0.66743402,0.77584183){\color[rgb]{0,0,0}\makebox(0,0)[lb]{\smash{${\mathbf D}_1$}}}%
    \put(0.08166361,0.64994493){\color[rgb]{0,0,0}\makebox(0,0)[lb]{\smash{${\mathbf D}_0$}}}%
    \put(0.88950219,0.19181995){\color[rgb]{0,0,0}\makebox(0,0)[lb]{\smash{$C_\pi$}}}%
    \put(0.89999368,0.88949879){\color[rgb]{0,0,0}\makebox(0,0)[lb]{\smash{$C_{-\pi}$}}}%
    \put(0.07292077,0.82305319){\color[rgb]{0,0,0}\makebox(0,0)[lb]{\smash{$C_{-0}$}}}%
    \put(0.08865789,0.28449411){\color[rgb]{0,0,0}\makebox(0,0)[lb]{\smash{$C_{+0}$}}}%
  \end{picture}%
\endgroup

%% file: figs/dulacmaps.pspdftex
\begin{picture}(0,0)%
\includegraphics{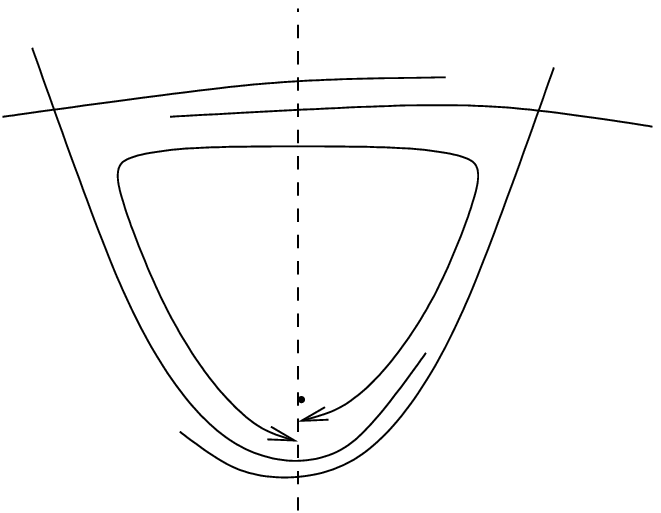}%
\end{picture}%
\setlength{\unitlength}{4144sp}%
\begingroup\makeatletter\ifx\SetFigFontNFSS\undefined%
\gdef\SetFigFontNFSS#1#2#3#4#5{%
  \reset@font\fontsize{#1}{#2pt}%
  \fontfamily{#3}\fontseries{#4}\fontshape{#5}%
  \selectfont}%
\fi\endgroup%
\begin{picture}(2994,2319)(6469,-3178)
\put(7111,-1951){\makebox(0,0)[lb]{\smash{{\SetFigFontNFSS{12}{14.4}{\rmdefault}{\mddefault}{\updefault}{\color[rgb]{0,0,0}$\dul_1$}%
}}}}
\put(8281,-1906){\makebox(0,0)[lb]{\smash{{\SetFigFontNFSS{12}{14.4}{\rmdefault}{\mddefault}{\updefault}{\color[rgb]{0,0,0}$\dul_2$}%
}}}}
\put(7831,-2581){\makebox(0,0)[lb]{\smash{{\SetFigFontNFSS{12}{14.4}{\rmdefault}{\mddefault}{\updefault}{\color[rgb]{0,0,0}$p_c(\ep,\delta)$}%
}}}}
\end{picture}%

%% file: figs/domain.pdf_tex
%% Creator: Inkscape inkscape 0.48.1, www.inkscape.org
%% Accompanies image file 'domain.pdf' (pdf, eps, ps)

\begingroup
  \makeatletter
  \providecommand\color[2][]{%
    \errmessage{(Inkscape) Color is used for the text in Inkscape, but the package 'color.sty' is not loaded}
    \renewcommand\color[2][]{}%
  }
  \providecommand\transparent[1]{%
    \errmessage{(Inkscape) Transparency is used (non-zero) for the text in Inkscape, but the package 'transparent.sty' is not loaded}
    \renewcommand\transparent[1]{}%
  }
  \providecommand\rotatebox[2]{#2}
  \ifx\svgwidth\undefined
    \setlength{\unitlength}{552.83681641pt}
  \else
    \setlength{\unitlength}{\svgwidth}
  \fi
  \global\let\svgwidth\undefined
  \makeatother
  \begin{picture}(1,0.84737467)%
    \put(0,0){\includegraphics[width=\unitlength]{domain.pdf}}%
    \put(0.11431433,0.44228201){\color[rgb]{0,0,0}\makebox(0,0)[lb]{\smash{$y_c$}}}%
    \put(0.01402512,0.83418135){\color[rgb]{0,0,0}\makebox(0,0)[lb]{\smash{$C_{-\pi}^1$}}}%
    \put(0.00631057,0.69531939){\color[rgb]{0,0,0}\makebox(0,0)[lb]{\smash{$C_{-\pi}^2$}}}%
    \put(0.01556803,0.12598529){\color[rgb]{0,0,0}\makebox(0,0)[lb]{\smash{$C_{\pi}^2$}}}%
    \put(0.01556803,0.00409533){\color[rgb]{0,0,0}\makebox(0,0)[lb]{\smash{$C_{\pi}^1$}}}%
    \put(0.25934796,0.47776894){\color[rgb]{0,0,0}\makebox(0,0)[lb]{\smash{${\bf D }$}}}%
    \put(0.79151329,0.46595909){\color[rgb]{0,0,0}\makebox(0,0)[lb]{\smash{$s_1$}}}%
    \put(0.91186033,0.45207288){\color[rgb]{0,0,0}\makebox(0,0)[lb]{\smash{$s_2$}}}%
  \end{picture}%
\endgroup

%% file: figs/domain1.pdf_tex
%% Creator: Inkscape inkscape 0.48.1, www.inkscape.org
%% Accompanies image file 'domain1.pdf' (pdf, eps, ps)

\begingroup
  \makeatletter
  \providecommand\color[2][]{%
    \errmessage{(Inkscape) Color is used for the text in Inkscape, but the package 'color.sty' is not loaded}
    \renewcommand\color[2][]{}%
  }
  \providecommand\transparent[1]{%
    \errmessage{(Inkscape) Transparency is used (non-zero) for the text in Inkscape, but the package 'transparent.sty' is not loaded}
    \renewcommand\transparent[1]{}%
  }
  \providecommand\rotatebox[2]{#2}
  \ifx\svgwidth\undefined
    \setlength{\unitlength}{784.90292969pt}
  \else
    \setlength{\unitlength}{\svgwidth}
  \fi
  \global\let\svgwidth\undefined
  \makeatother
  \begin{picture}(1,0.59683802)%
    \put(0,0){\includegraphics[width=\unitlength]{domain1.pdf}}%
    \put(0.42290749,0.31368941){\color[rgb]{0,0,0}\makebox(0,0)[lb]{\smash{$y_c$}}}%
    \put(0.3055406,0.58754547){\color[rgb]{0,0,0}\makebox(0,0)[lb]{\smash{$C_{-\pi}^1$}}}%
    \put(0.30010695,0.48973974){\color[rgb]{0,0,0}\makebox(0,0)[lb]{\smash{$C_{-\pi}^2$}}}%
    \put(0.30662733,0.0887362){\color[rgb]{0,0,0}\makebox(0,0)[lb]{\smash{$C_{\pi}^2$}}}%
    \put(0.30662733,0.0028845){\color[rgb]{0,0,0}\makebox(0,0)[lb]{\smash{$C_{\pi}^1$}}}%
    \put(0.47507056,0.31477611){\color[rgb]{0,0,0}\makebox(0,0)[lb]{\smash{${\bf D }$}}}%
    \put(0.85315492,0.32819259){\color[rgb]{0,0,0}\makebox(0,0)[lb]{\smash{$s_1$}}}%
    \put(0.93791989,0.31841201){\color[rgb]{0,0,0}\makebox(0,0)[lb]{\smash{$s_2$}}}%
    \put(0.29348872,0.31515184){\color[rgb]{0,0,0}\makebox(0,0)[lb]{\smash{$O(\varepsilon)$}}}%
    \put(0.61081402,0.41730449){\color[rgb]{0,0,0}\makebox(0,0)[lb]{\smash{$y$}}}%
    \put(0.61190075,0.16409627){\color[rgb]{0,0,0}\makebox(0,0)[lb]{\smash{$\bar{y}$}}}%
    \put(-0.00101525,0.38470259){\color[rgb]{0,0,0}\makebox(0,0)[lb]{\smash{$\dul_1(y)= \dul_1(\bar{y})$}}}%
    \put(0.01419898,0.17931054){\color[rgb]{0,0,0}\makebox(0,0)[lb]{\smash{$\dul_2(y)= \dul_2(\bar{y})$}}}%
  \end{picture}%
\endgroup